\documentclass[twoside]{article}
\usepackage{amsfonts, amsmath}
\usepackage{xcolor}
\usepackage{amssymb}
\usepackage{amsfonts}
\usepackage{microtype}
\usepackage{enumitem}
\usepackage{amsmath,amssymb,amsfonts}%
\usepackage{amsthm}%
\usepackage{mathrsfs}%
\usepackage[colorlinks=true]{hyperref}
\hypersetup{linkcolor=blue}
\usepackage[dvips, lmargin=2.5cm, rmargin=1.5cm, tmargin=2.5cm, bmargin=2cm]{geometry}

\newtheorem{theorem}{Theorem}
\newtheorem{lemma}{Lemma}

\newtheorem{remark}{Remark}%

\newtheorem{definition}{Definition}

\newcommand{\verti}[1]{{\left\vert\kern-0.25ex\left\vert\kern-0.25ex\left\vert #1
\right\vert\kern-0.25ex\right\vert\kern-0.25ex\right\vert}}
\newcommand{\vertiii}[1]{{\left\vert\kern-0.25ex\left\vert\kern-0.25ex\left\vert #1
    \right\vert\kern-0.25ex\right\vert\kern-0.25ex\right\vert}}
\title{
\textbf{Approximate Controllability of Fractional Differential Systems with Nonlocal Conditions of Order $q\in ]1,2[$ }}

\numberwithin{equation}{section}

\date{}
\author{A. Aberqi$^{1}$, Z.. Ech-chaffani$^{2}$ and T. Karite$^{1}$ \\
$^{1}$Sidi Mohamed ben abdellah university, National School of Applied Sciences\\, Fez, Morocco\\
\href{www.google.com}{\color{blue}ahmed.aberqi@usmba.ac.ma, touria.karite@usmba.ac.ma}
\\
$^{2}$Faculty of Sciences Dhar El Mahraz, Sidi Mohamed Ben Abdellah University,\\ Fez Morocco.
$^{2}$  \href{http//:www.gmail.com}{\color{blue}echzoubida@gmail.com}$^{2}$}

\everymath{\displaystyle}
\begin{document}
\maketitle
\vspace{0.5cm}
\begin{abstract}
	In this manuscript, we are concerned with the approximate controllability of fractional nonlinear differential equations with nonlocal conditions of order $1<q<2$ in  Banach spaces. As far as we know, few articles have investigated this issue. The idea is to see under which sufficient conditions the proposed control problem is approximately controllable. The discussion is based on the theory of resolvent operator, fractional calculus techniques and Krasnoselskii's fixed point theorem under the assumption that the associated linear system is approximately controllable. The obtained results improve some existing analogous ones on this topic. Finally, an example is provided to illustrate the applications of the obtained results
\end{abstract}
{\bf AMS:}	35J60,  58J05, 35R11, 35J75, 35J60, 46E35.
\section{Introduction}
In recent years, fractional differential equations have garnered increasing interest because they efficiently describe various phenomena arising in physics, engineering, economics and other scientific domains more accurately than integer-order equations \cite{Byszewski, Anguraj,kumar2023legendre}.  Since fractional derivatives offer an exceptional tool for characterizing memory and hereditary effects in diverse materials and processes. In recent decades, there has been a growing interest among scientific researchers in the utilization of fractional differential equations. This surge in interest stems from notable advancements in both the theory and practical applications of these equations, for instance, \cite{Baleanu, Samko, Dineshkumar} and the references therein.

The concept of controllability is one of the most powerful concepts of control theory. This is the qualitative property of control theory systems and is of particular importance within the field. It was introduced first by Kalman \cite{Kalman} in 1963.  The Controllability problem is to find an appropriate control function that allows us to steer the solution of a dynamic system from its initial state to a desired final state. However, many dynamical systems are structured in such a way that the control influences only a portion of the system's state rather than its entirety. Similarly, in many real-world industrial processes, only a subset of the system's full state can be observed. As a result, it becomes crucial to assess whether controlling the whole state of the system is feasible. This is where the concept of approximate controllability arises. Controllability for various types of integer order differential systems and fractional differential systems of order $q\in (0,1)$ in Banach spaces using several different approaches is well developed and has been extensively studied (see \cite{Curtain, Balachandran, Klamka2, Naito, Lin2021sliding, Park, Aberqi, Ech-chaffani, Karite,uematsu2016data,yang2023fixed,yinoptimal, Alami} ) and the references therein.  However,  few studies focussed on approximate controllability for fractional problems with order $q\in(1,2)$. Shukla et al in \cite{Shukla} proved the approximate controllability of the following fractional order semilinear system of order $q\in (1,2]$ 
	\begin{equation*}
		\begin{cases}
			^{c}D^{q}z(\theta) =^{}-\mathcal{A} z(\theta)+f(\theta,z(\theta))+Bu(\theta), & \theta \in J=[0, a], \\
			\displaystyle   z(0)=z_{0},\\
			z^{\prime}(0=z_{1},
		\end{cases}
	\end{equation*}
where $^{c}D^{q}$ is the Caputo derivative, $\mathcal{A}:D(\mathcal{A})\subset Z\to Z$ is a closed linear operator which generates a strongly continuous $q-$ order cosine family on a Hilbert space. The results are obtained by using Schauder's fixed point theorem, the authors in \cite{shu2019} established suitable assumptions to guarantee the existence and uniqueness of mild solutions and proved under these conditions the approximate controllability of the associated fractional control system involving Riemann--Liouville fractional derivatives with order $q\in (1,2)$. Kumar \cite{kumar2022} discussed the approximate controllability for fractional neutral integro-differential inclusions with non-instantaneous impulses and infinite delay using the fixed point theorem for discontinuous multi-valued operators with the q-resolvent operator. Recently, Dineshkumar et al \cite{dineshkumar2023new} studied the approximate controllability of the following Sobolev type fractional stochastic evolution hemivariational inequalities of order $q\in (1,2)$:
	\begin{equation*}
		\begin{cases}
			^{c}D^{q}[Jz(\theta)] =^{}-\mathcal{A} z(\theta)+F(\theta,z(\theta))\displaystyle\frac{dw(\theta)}{d\theta}+\partial G(\theta, z(\theta))+Bu(\theta), & \theta \in J=[0, a], \\
			\displaystyle   z(0)=z_{0},\\
			z^{\prime}(0)=z_{1},
		\end{cases}
	\end{equation*}
    in Hilbert space by integrating fractional order calculations and using fixed-point theorems for multivalued mappings, cosine, and sine family operators. In \cite{shukla2016approximate} Shukla et al obtained some results for approximate controllability for fractional order system with infinite delay using the theory of $q-$ order cosine family and sequential. In \cite{shukla2017approximate}, the authors, established the approximate controllability for the following fractional stochastic semilinear control system of order $q\in (1,2)$ in Banach space using the theory of strongly continuous  $q-$ order cosine family, fixed point theorem and stochastics' analytiques techniques. Most of the above papers are with initial standard conditions. Moreover, to describe some physical phenomena more precisely, the nonlocal conditions can be more useful than classical initial conditions $z(0)=z_{0}$. For example, in\cite{deng1993exponential} the nonlocal conditions are used to describe the diffusion phenomena. For more applications a detailed review is presented in Diethelm and Ford 2004 \cite{Diethelm}. To the best of our knowledge, most of the existing results on approximate controllability of order $q\in (1,2)$ are obtained with standard initial conditions. Limited work has been done for order $q\in (1,2)$ with nonlocal conditions and most of them are in Hilbert spaces. Li et al in \cite{fan2016approximate} discussed the approximate controllability of a class of semilinear systems composite relaxation equations in Hilbert space with nonlocal conditions under the assumption that the corresponding linear system is approximately controllable. The authors in \cite{lian2017approximate} investigated the approximate controllability for a fractional semilinear differential system of order $q\in (1,2)$ with nonlocal conditions in Hilbert space. Yang in \cite{yang2021approximate} proved the approximate controllability of the Sobolev type fractional control system with nonlocal conditions in Hilbert space. However, It should be mentioned that the approximate controllability of fractional nonlinear equations with nonlocal conditions in Banach space is still in the initial stage. Motivated by this observation. In the present work, we will discuss the approximate controllability of the following nonlinear fractional control system of order $q\in (1,2)$ in Banach space:
\begin{equation}\label{system1}
	\begin{cases}
		^{c}D^{q}z(\theta) =^{}-\mathcal{A} z(\theta)+f(\theta,z(\theta),\displaystyle\int_{0}^{\theta}g(\theta,\tau,z(\tau))d\tau)+Bu(\theta), & \theta \in J=[0, a], \\
		\displaystyle   z(0)+\phi(z)=z_{0},\\
		z^{\prime}(0)+\psi(z)=z_{1},
	\end{cases}
\end{equation}
where $^{c}D^{q}$ is the Caputo fractional derivative, $\mathcal{A}$ is the infinitesimal generator of a strongly continuous $q$-order cosine family $\{C_{q}(\theta)\}_{\theta\geq 0}$ on a Banach space $X$, the state $z(\cdot)$ takes values in $X$, $z_{0}$, $z_{1}\in Z$ the control function $u(\cdot)$ is given in $L^{2}([0, a];U)$, where $U$ is a Banach space, $B$ is a bounded linear operator from $U$ into $X$, $f,g,\psi,\phi$ are given continuous functions to be specified later.
Our work distinguishes itself from other ones in the following aspects:
\begin{itemize}
	\item[(i)] The system \eqref{system1} represents a nonlinear fractional differential system with order $q\in (1,2)$  and with nonlocal conditions in a Banach space.
	\item[(ii)] A complexity arises from the existence of nonlocal conditions and the existence of an integral term in the function $f$. 
 \item[(iii)] The approximate controllability of system \eqref{system1} is established. A nonlinear fractional control system with nonlocal conditions is used as an example to illustrate the application of the obtained results.
\end{itemize}

The remainder of this paper is structured as follows. Section~\ref{sec2}, offers fundamental definitions and facts that underpin our analysis. In Section~\ref{sec3}, we focus on the approximate controllability of the system \eqref{system1}. Finally, an example is provided to illustrate the main results in section~\ref{sec4}.

\section{Preliminaries}\label{sec2}

Throughout this paper, $Z$ will be a Banach space equipped with the norm $\|\cdot\|$ and $C(J, Z)$ is the space of continuous functions from $J$ to $Z$ provided with the supremum norm $\|\cdot\|_{C}$.
To begin with the analysis we need some basic definitions and properties from the fractional calculus theory. 
 Let us recall the following known definitions. For more details see \cite{Podlubny}.
\begin{definition}
	The fractional integral of order $q>0$ with the lower limit zero for a function $z$ is  defined as 
	\begin{equation}
		I^{q}z(\theta)=\displaystyle\frac{1}{\Gamma(q)}\displaystyle\int_{0}^{\theta}(\theta-\nu)^{q-1}z(\nu)d\nu,\quad \theta>0,
	\end{equation} 
provided the right-hand side is pointwise defined on $[0,\infty($, where $\Gamma(\cdot)$ is the gamma function.
\end{definition}
 \begin{definition}
     The Caputo derivative of order $q>0$ with the lower limit zero for a function $z$ can be written as 
     \begin{equation}
         ^{C}D^{q}z(\theta)=\displaystyle\frac{1}{\Gamma(n-q)}\displaystyle\int_{0}^{\theta}(\theta-\nu)^{n-q-1}z^{n}(\nu)d\nu=I^{n-q}z^{(n)}(\theta),\quad \theta>0,\quad n-1<q<n.         
     \end{equation}
 \end{definition}
 Let the linear fractional order system as
\begin{equation}\label{system}
	\begin{cases}
		^{c}D^{q}z(\theta) =\mathcal{A} z(\theta)& \theta \in J=[0, a], 1<q< 2, \\
		\displaystyle   z(0)=z_{0},\\
		z^{\prime}(0)=z_{1},
	\end{cases}
\end{equation}
where, $\mathcal{A}:D(\mathcal{A})\subset Z\to Z$ is a closed and dense operator in $Z$. Using the fractional integral of order $q$ on \eqref{system}, we get
\begin{equation*}
	z(\theta)=z_{0}+\displaystyle\frac{1}{\Gamma(q)}\displaystyle\int_{0}^{\theta}(\theta-\nu)^{q-1}\mathcal{A}z(\nu)d\nu.
\end{equation*}
\begin{definition}\cite{Bazhlekova}
	Let $1<q\leq 2$. A family $\{C_{q}(\theta)\}_{\theta\geq 0}\subset\mathcal{B}(Z)$ is termed a solution operator (or a strongly continuous $q$-order fractional cosine family) for \eqref{system} if it meets the following conditions:
 \begin{itemize}
		\item [(i)] $C_{q}(\theta)$ is strongly continuous for $\theta\geq 0$ and $C_{q}(0)=I$;
		\item [(ii)] $C_{q}(\theta)D(\mathcal{A})\subset D(\mathcal{A})$ and $\mathcal{A}C_{q}(\theta)\xi=C_{q}(\theta)\mathcal{A}\xi$ for all $\xi\in D(\mathcal{A})$, $\theta\geq 0$;
		\item [(iii)] $C_{q}(\theta)\xi$ is a solution of \eqref{system} for all $\xi \in D(\mathcal{A})$.\\
		$\{S_{q}(\theta)\}\theta\geq 0$ is the sine family associated with the strongly continuous $q-$ order fractional cosine family $\{C_{q}(\theta)\}_{\theta\geq 0}$ which is defined by$
			S_{q}(\theta)z=\displaystyle\int_{0}^{\theta}C_{q}(\nu)zd\nu,\, \theta\geq 0.$
		$\mathcal{A}$ is called the infinitesimal generator of $C_{q}(\theta)$.
	\end{itemize}
\end{definition}
\begin{definition}\cite{Bazhlekova}
The $q$-order cosine family $C_{q}(\theta)$ is referred to as exponentially bounded if there exist constants $M\geq 1$ and $\eta\geq 0$ such that
	\begin{equation}
		\|C_{q}(\theta)\|\leq Me^{\eta \theta},\quad \theta\geq 0.
	\end{equation}
\end{definition}
\begin{definition}\cite{Shukla}
	The Riemann--Liouville fractional family $P_{q}:\mathbb{R}^{+}\to \mathcal{L}(X)$ associated with $C_{q}$ is defined by
	\begin{equation}
		P_{q}(\theta) =	I^{q-1}C_{q}(\theta).
	\end{equation}
\end{definition}

\begin{lemma}\label{lemmaM}
	For any $\theta\in [0, a]$ and $z\in Z$, we have
	\begin{equation}
		\|P_{q}(\theta)z\|\leq \displaystyle\frac{Ma^{q-1}}{\Gamma(q)}\|z\|.
	\end{equation}
\end{lemma}

\begin{proof}
	For any $z\in Z$ and $\theta\in [0, a]$, we have
	\begin{equation*}
		\|P_{q}(\theta)x\| =\|I^{q-1}C_{q}(\theta)x\|=\Big\|\displaystyle\frac{1}{\Gamma(q-1)}\displaystyle\int_{0}^{\theta}(\theta-\nu)^{(q-2)}C_{q}(\nu)zd\nu\Big\| 
		\leq\displaystyle\frac{M}{\Gamma(q-1)}\|z\| \Big\|\displaystyle\int_{0}^{\theta}(\theta-\nu)^{(q-2)}d\nu\Big\|\leq \displaystyle\frac{Ma^{q-1}}{\Gamma(q)}\|z\|.
	\end{equation*}
\end{proof}
 To define the concept of mild solution for the control system \eqref{system1}, by comparison with the fractional differential equation given in \cite{Kexue}, we associate system \eqref{system1} to the integral equation.
\begin{definition}
	A function $z(\cdot)\in C([0, a],Z)$ is said to be a mild solution of \eqref{system1} if the following integral equation is satisfied
	\begin{equation}\label{mild solution}
		z(\theta)=C_{q}(\theta)[z_{0}-\phi(z)]+S_{q}(\theta)[z_{1}-\psi(z)]+\displaystyle\int_{0}^{\theta}P_{q}(\theta-\nu)f(\nu,z(\nu),\displaystyle\int_{0}^{\nu}g(\nu,\tau,z(\tau))d\tau)d\nu+\displaystyle\int_{0}^{\theta}P_{q}(\theta-\nu)Bu(\nu)d\nu.
	\end{equation} 
\end{definition}
 Next, we recall the following definition for approximate controllability of \eqref{system1}.
\begin{definition}\cite{Curtain}
    The control system \eqref{system1} is said to be approximately controllable on $[0, a]$ if for every $z_{0},z_{1}\in Z$, there exists some control $u\in L^{2}(0,a,U) $, the closure of the reachable set is dense in $X$, that is   $\overline{\mathcal{R}(a;z_{0},z_{1},u)}=Z,$ where the reachable set is defined as
    \begin{equation*}
    	\mathcal{R}(a;z_{0},z_{1})=\{z(a, z_{0},z_{1},u),\quad u(\cdot)\in L^{2}(J,U)\}.
    \end{equation*} 
\end{definition}
Consider the linear fractional control system
\begin{equation}\label{linearsys}
	\begin{cases}
		^{c}D^{q}z(\theta) =^{}-A z(\theta)+Bu(\theta), & \theta \in J=[0, a], \\
		\displaystyle   z(0)=z_{0},\\
		z^{\prime}(0)=z_{1},
	\end{cases}
\end{equation} 
 The approximate controllability for linear fractional control system \eqref{linearsys} is a natural generalisation of approximate controllability of linear first order control system \cite{Mahmudov}. It is convenient at this point to introduce the following operators associated with \eqref{linearsys} as
\begin{equation*}
	\mathbf{K}_{0}^{a}=\displaystyle\int_{0}^{a}P_{q}(a-\nu)BB^{\star}P_{q}^{\star}(a-\nu)d\nu:\quad Z\to Z,\quad \quad \mathbf{R}(\beta,\mathbf{K}_{0}^{a})=(\beta I+\mathbf{K}_{0}^{a})^{-1}:\quad Z\to Z
\end{equation*}
here $\beta>0$, $I$ is the identity operator, $B^{\star}$ denotes the adjoint of $B$ and $P_{q}^{\star}$ is the adjoint of $P_{q}$.  It is straightforward that the $\mathbf{K}_{0}^{a}$ is a linear bounded operator and $\mathbf{R}(\beta,\mathbf{K}_{0}^{T})$ is a continuous operator with $\|\mathbf{R}(\beta,\mathbf{K}_{0}^{T})\|\leq \displaystyle\frac{1}{\beta}$ (see \cite{N.I}).\\
We conclude this section by recalling Krasnoselskii's Fixed point Theorem \cite{Karakostas} which will be used in the sequel to prove the existence of mild solutions of system \eqref{system1}.
\begin{theorem}\label{Kras}(Krasnoselskii's Fixed point Theorem) . 
Let $Z$ be a Banach space, $\mathcal{B}$ a bounded, closed, and convex subset of $X$, and $\mathcal{G}_{1}$, $\mathcal{G}_{2}$ be mappings from $\mathcal{B}$ to $X$ such that $\mathcal{G}_{1}z+\mathcal{G}_{2}w\in\mathcal{B}$ for every $z,w\in \mathcal{B}$. If $\mathcal{G}_{1}$ is a contraction and $\mathcal{G}_{2}$ is completely continuous, then the equation $\mathcal{G}_{1}z+\mathcal{G}_{2}z=z$ has a solution in $\mathcal{B}$.
\end{theorem}

\section{Main result}\label{sec3}

In this section,  we study the approximate controllability of the control system \eqref{system1}. First, we show that, for any $z_{d}\in Z$, by choosing proper control $u_{\beta}$ (for any $\beta \in (0,1)$), there exists a mild solution $z(\cdot, z_{0},z_{1},u_{\beta})\in C([0, a],  Z)$ for system \eqref{system1}, and then we prove that under certain conditions, approximate controllability of the linear system \eqref{linearsys} implies the approximate controllability of the nonlinear control system \eqref{system1}. 
\subsection{Assumptions on data}
To study the mild solution, we need the following assumptions:
\begin{itemize}
	\item[$(\mathcal{H}_{1})$] The uniformly bounded linear cosine family $\{C_{q}(\theta)\}_{\theta\geq 0}$ generated by $A$ is compact and continuous in the uniform operator topology for any $\theta\in[0, a]$. 
	\item[$(\mathcal{H}_{2}):$] The function $f:J\times Z\times Z\to Z$ satisfies the following conditions:
	\item [(i)] $f(\cdot,z,w)$ is measurable for all $(z,w)\in Z\times Z$ and $f(\theta,\cdot,\cdot)$ is continuous for a.e. $\theta\in J.$
	\item [(ii)] There exist constants $C_{1}, C_{2}$  such that 
	$\quad\|f(\theta,z,w)-f(\theta,z^{\prime},w^{\prime})\|\leq C_{1}\|z-z^{\prime}|\|+C_{2}\|w-w^{\prime}\|$, for every $ (\theta,z,w),(\theta,z^{\prime},w^{\prime})\in J\times Z\times Z.$
	\item[(iii)] There exist positive function $m\in L^{\infty}([0, a],\mathbb{R}^{+})$ such that  $\|f(\theta,z(\theta),\displaystyle\int_{0}^{\theta}g(\nu,\tau,z(\tau)d\tau)\|\leq m(\theta)$ for all $\theta\in [0, a]$.
	\item[$(\mathcal{H}_{3})$:] The function $ g:J\times J\times Z\to Z$ satisfies:
	\item[(i)] $g(\cdot,\nu,z)$  is measurable for all $(\nu,z)\in J\times Z$ and $g(\theta,\cdot,\cdot)$ is continuous for a.e. $\theta\in J$.
	\item[(ii)] There exists a constant $C_{3}>0$ such that 
	$\|g(\theta,\nu,z)-g(\theta,\nu,z^{\prime})\|\leq C_{3}\|z-z^{\prime}\|,\quad \forall \theta,\nu\in J, z,z^{\prime}\in Z.$
	\item[$(\mathcal{H}_{4})$ ] $\phi,\psi: X\to \overline{D(\mathcal{A})}$ is continuous, and there exist constants $d_{1}$, $d_{2}$ such that 
	$\|\phi(z)-\phi(w)\|\leq d_{1}\|z-w\|$, and $\|\psi(z)-\psi(w)\|\leq d_{2}\|z-w\|$, for $z,w\in Z$.
\end{itemize}
\subsection{Example}
In this section we give an example of functions that satisfy the assumptions  $(\mathcal{H}_{2})$, $(\mathcal{H}_{3})$ and $(\mathcal{H}_{4})$. Let $Z=L^{2}([0,\pi])$, $J=[0, 1]$, and the function $g:J\times J\times Z\to Z$ is defined by
$$
g(\theta,\nu,\varphi)=\displaystyle\frac{e^{\nu}}{\sqrt{2}+|\varphi(\nu,y)|},\quad y\in [0,\pi],
$$
 the function $f:J\times Z\times Z\to Z$ is given by
 \begin{equation*}
 	\begin{array}{lll}
 		f(\theta,\varphi,\displaystyle\int_{0}^{\theta}g(\theta,\nu,\varphi)d\nu)(y)=\displaystyle\frac{e^{-\theta}|\varphi(\theta,y)|}{(3+e^{\theta})(1+|\varphi(\theta,y)|)}+\displaystyle \int_{0}^{\theta}\displaystyle\frac{e^{\nu}}{\sqrt{2}+|\varphi(\nu,y)|}d\nu,\quad y\in [0,\pi],\\\\
 		\phi(z)(y)=\displaystyle\sum_{i=1}^{i=p}a_{i
 		}z(\theta)(y),\quad \quad\qquad \qquad\qquad  \psi(z)(y)=\displaystyle\sum_{i=1}^{i=p}\gamma_{i
 		}z(\theta)(y)
 	\end{array}
 \end{equation*}
where $p$ is a positive integer, $a_{i}$, $\gamma_{i}>$ and we assume that there exists $L_{1},~L_{2}>0$ such that 
$
\displaystyle\sum_{i=1}^{p}|a_{i}|<L_{1}$ and $\displaystyle\sum_{i=1}^{p}|\gamma_{i}|<L_{2}$.
In fact, for any $\varphi,\varphi^{\prime}\in Z$, we have
$$
\|g(\theta,\nu,\varphi)-g(\theta,\nu,\varphi^{\prime})\|\leq \Big\|\displaystyle\frac{e^{\nu}}{\sqrt{2}+\varphi}-\displaystyle\frac{e^{\nu}}{\sqrt{2}+\varphi^{\prime}}\Big\|\leq \displaystyle\frac{e}{2}\|\varphi-\varphi^{\prime}\|,
$$
which implies that $(\mathcal{H}_{3})-(ii)$ holds.
To show $(\mathcal{H}_{2})-(ii)$ we proceed as follows:
$$
\Big\|f(\theta,\varphi,\displaystyle\int_{0}^{\theta}g(\theta,\nu,\varphi)d\nu))-f(\theta,\varphi^{\prime},\displaystyle\int_{0}^{\theta}g(\theta,\nu,\varphi^{\prime})d\nu))\Big\|\leq\displaystyle\frac{e^{-\theta}}{3+e^{\theta}}\|\varphi-\varphi^{\prime}\|+\displaystyle\frac{e}{2}\|\varphi-\varphi^{\prime}\|\leq \displaystyle\frac{2+3e}{6}\|\varphi-\varphi^{\prime}\|.
$$
Thus $(\mathcal{H}_{2})-(ii)$ holds. And obviously, $(\mathcal{H}_{2})-(ii)$ is satisfied. Moreover, for any $\varphi,\varphi^{\prime}\in Z$, we have
$$
\|\phi(\vartheta)-\phi(\vartheta^{\prime})\|\leq \sum_{i=1}^{p}|a_{i}|\|\vartheta-\vartheta^{\prime}\|\leq L_{1}\|\vartheta-\vartheta^{\prime}\|,\quad\quad\quad \|\psi(\vartheta)-\psi(\vartheta^{\prime})\|\leq \sum_{i=1}^{p}|\gamma_{i}|\|\vartheta-\vartheta^{\prime}\|\leq L_{2}\|\vartheta-\vartheta^{\prime}\|
$$
Hence, the assumption $(\mathcal{H}_{4})$ holds.
Now, we are positioned to discuss the existence of mild solution of the fractional system \eqref{system1}.
\subsection{Existence of mild solution}
For any $z_{d}\in Z$ and $\beta\in (0,1)$, we choose the feedback control 
\begin{equation}\label{control}
		\begin{array}{ll}
		u_{\beta}(\theta,z)&=B^{\star}P_{q}^{\star}(a-\theta)(\beta I+\mathbf{K}_{0}^{a})^{-1}\bigg(z_{d}-C_{q}(a)(z_{0}-\phi(z))-S_{q}(a)(z_{1}-\psi(z))\bigg) \\&-B^{\star}P_{q}^{\star}(a-\theta)\displaystyle\int_{0}^{a}(\beta I+\mathbf{K}_{0}^{a})^{-1}P_{q}(a-\nu)f(\nu,z(\nu),\int_{0}^{\nu}g(\nu,\tau,z(\tau))d\tau)d\nu. 
		\end{array}	
		\end{equation}
  It will be shown that by using the above control \eqref{control} the operator $\mathcal{G}: C([0, a],Z)\to C([0, a],Z)$ defined by:
		$$(\mathcal{G}z)(t)=C_{q}(\theta)[z_{0}-\phi(z)]+S_{q}(\theta)[z_{1}-\psi(z)]+\displaystyle\int_{0}^{\theta}P_{q}(\theta-\nu)f(\nu,z(\nu),\displaystyle\int_{0}^{\nu}g(\nu,\tau,z(\tau))d\tau)d\nu+\displaystyle\int_{0}^{\theta}P_{q}(\theta-\nu)Bu_{\beta}(\nu,z)d\nu.$$
  has at least one fixed point $z(\cdot)$ in $C([0, a], Z)$. This fixed point is then a mild solution of the control system \eqref{system1}. We will employ Krasnoselskii's Fixed Point Theorem to prove this. 
		The following lemma will be useful to prove that $\mathcal{G}(B_{k})\subset B_{k}$, where $B_{k}$ is a closed, convex and bounded subset of $C([0, a], X)$ which will be defined in the sequel.
		 
  \textbf{Lemma 2}\label{lemmaB}
	 If $(\mathcal{H}_{1})$-$(\mathcal{H}_{3})$ are satisfied, then there exist positive constants $\mathcal{L}_{4}$ and $\mathcal{L}_{5}$ such that for all $z\in C([0, a],  Z)$ and $\theta\in [0, a]$ we have
 	\begin{equation}\label{estim control}
	\|u_{\beta}(\theta,z)\|\leq\displaystyle\frac{1}{\beta}\Big(\mathcal{L}_{4}+\mathcal{L}_{5}\|z\|_{C}\Big)
	\end{equation}

 \begin{proof}
 
     {In view of \eqref{control} and lemma~\ref{lemmaM}, for $\theta\in [0, a]$, we obtain
     		\begin{equation*}
     			\begin{array}{llll}
     				\|u_{\beta}(\theta,z)\|&\leq\Big\|B^{\star}P_{q}^{\star}(a-\theta)(\beta I+\mathbf{K}_{0}^{a})^{-1}\Big[z_{d}-C_{q}(a)(z_{0}-\phi(z))-S_{q}(a)(z_{1}-\psi(z))-\displaystyle\int_{0}^{a}P_{q}(a-\nu)f(\nu,z(\nu), \int_{0}^{\nu}\\&g(\nu,\tau,z(\tau))d\tau)d\nu\Big]\Big\|\\
     				&\leq \displaystyle\frac{M_{B}M a^{q-1}}{\beta\Gamma(q)} \Big\|z_{d}-C_{q}(a)(z_{0}-\phi(z))-S_{q}(a)(z_{1}-\psi(z))-\displaystyle\int_{0}^{a}P_{q}(a-\nu)f(\nu,z(\nu),
     				\int_{0}^{\nu}g(\nu,\tau,z(\tau))d\tau)d\nu\Big\|\\
     				&\leq \displaystyle\frac{M_{B}M a^{q-1}}{\beta\Gamma(q)}\Bigg[\|z_{d}\|+M\|z_{0}-\phi(z)\|+M\|z_{1}-\psi(z)\|+\displaystyle\frac{Ma^{q-1}}{\Gamma(q)}\Bigg(\displaystyle\int_{0}^{a}\|f(\nu,z(\nu),\int_{0}^{\nu}g(\nu,\tau,z(\tau))d\tau)\\
     				&-\displaystyle f(\nu,0,\int_{0}^{\nu}g(\nu,\tau,0)d\tau)\|d\nu+\displaystyle\int_{0}^{a}\|f(\nu,0,\int_{0}^{\nu}g(\nu,\tau,0)d\tau)\|d\nu\Bigg)\Bigg]
     			\end{array}
     		\end{equation*}
 As a consequence of $(\mathcal{H}_{2})-(ii),(iii)$ and $(\mathcal{H}_{3})-(ii)$, we get}
 \begin{equation}
     \begin{array}{lll}
         \|u_{\beta}(\theta,z)\|
	&\leq \displaystyle\frac{M_{B}M a^{q-1}}{\beta\Gamma(q)}\Bigg[\|z_{d}\|+M[\|z_{0}\|+\|\phi(z)\|]+M[\|z_{1}\|+\|\psi(z)\|]+\displaystyle\frac{Ma^{q-1}}{\Gamma(q)}\Big(\displaystyle\int_{0}^{a}\Bigg
	(C_{1}\|z(\nu)\|\\
&+\displaystyle C_{2}C_{3}\int_{0}^{\nu}\|z(\tau)\|d\tau\bigg)d\nu+\displaystyle\int_{0}^{a}m(\nu)d\nu\Bigg)\Bigg]\\&\leq \displaystyle\frac{M_{B}M a^{q-1}}{\beta\Gamma(q)}\Bigg[\|z_{d}\|+M\overline{Y}+\displaystyle\frac{Ma^{q}}{\Gamma(q)}\Bigg(C_{1}\|z\|_{C}+aC_{2}C_{3}\|z\|_{C}+\|m\|_{L^{\infty}(J,\mathbb{R}^{+})}\Bigg)\Bigg]
     \end{array}
 \end{equation}
 Then  
 \begin{equation}
    \|u_{\beta}(\theta,z)\|\leq \displaystyle\frac{M_{B}}{\beta}\Big(\mathcal{L}_{4}+\mathcal{L}_{5}\|z\|_{C}\Big),
 \end{equation}
 where $\mathcal{L}_{4}=\displaystyle\frac{M_{B}Ma^{q-1}}{\Gamma(q)}\Big[\|z_{d}\|+M\overline{Y}+\displaystyle\frac{Ma^{q}}{\Gamma(q)}\|m\|_{L^{\infty}(J,\mathbb{R}^{+})}\Big]$, $\overline{Y}=(\|z_{0}\|+\|\phi(z)\|+\|z_{1}\|+\|\psi(z)\|)$,
 and

$\mathcal{L}_{5}=\displaystyle\frac{M_{B}M^{2}a^{2q-1}}{(\Gamma(q))^{2}}\Big[C_{1}+aC_{2}C_{3}\Big]$. The proof is complete.
  \end{proof}
   Under the above conditions, we prove the following existence theorem.
  \begin{theorem}\label{Thm mild} Assume that  $(\mathcal{H}_{1})-(\mathcal{H}_{4})$ are satisfied, Further suppose that $\{P_{q}(\theta)\}(\theta\geq 0)\}$ is compact, then the system \eqref{system1} has at least one mild solution on $[0, a]$ provided that  $M(d_{1}+d_{2})<1.$
\end{theorem}
\begin{proof}
    We base the argument on Theorem~\ref{Kras}. For each positive constant $k$, let $B_{k}=\{z\in C([0, a],Z): \|z\|_{C}\leq k\}$. It is evident that $B_{k}$ is a bounded, closed, and convex subset in $C([0, a],  Z)$. Define on $B_{k}$ the operators $\mathcal{G}_{1}$ and $\mathcal{G}_{2}$ by
	\begin{equation}\label{formG}
		\begin{cases}
			(\mathcal{G}_{1}z)(\theta)=C_{q}(\theta)[z_{0}-\phi(z)]+S_{q}(\theta)[z_{1}-\psi(z)], \\
			(\mathcal{G}_{2}z)(\theta)=\displaystyle\int_{0}^{\theta}P_{q}(\theta-\nu)f(\nu,z(\nu),\displaystyle\int_{0}^{\nu}g(\nu,\tau,z(\tau))d\tau)d\nu+\displaystyle\int_{0}^{\theta}P_{q}(\theta-\nu)Bu_{\beta}(\nu,z)d\nu,
		\end{cases}
	\end{equation}
  Clearly, the fixed point of the operator equation $\mathcal{G}:=\mathcal{G}_{1}+\mathcal{G}_{2}$ is a mild solution of \eqref{system1}. We will use Theorem~\ref{Kras} to prove that $\mathcal{G}$ has a fixed point. Our proof will be divided into three steps.\\
	\textbf{Step I: } $\mathcal{G}_{1}z+\mathcal{G}_{2}w\in B_{k}$, for every $z,w\in B_{k}$.
 From lemma ~\ref{lemmaM}, lemma~\ref{lemmaB}, $(\mathcal{H}_{2})-(ii)$, $(iii)$ and $(\mathcal{H}_{3})-(ii)$ it follows that for any $z,w\in B_{k}$
 \begin{equation*}
 	\begin{array}{llll}
 		\|(\mathcal{G}_{1}z)(\theta)+(\mathcal{G}_{2}w)(\theta)\|&\leq \|C_{q}(\theta)\|\|z_{0}-\phi(z)\|+\|S_{q}(\theta)\|\|z_{1}-\psi(z)\|+\displaystyle\frac{Ma^{q-1}}{\Gamma(q)}\\
 		&\Bigg(\displaystyle\int_{0}^{\theta}\|f(\nu,z(\nu),\int_{0}^{\nu}g(\nu,\tau,z(\tau))d\tau)-f(\nu,0,\int_{0}^{\nu}g(\nu,\tau,0)d\tau)\|d\nu\\
 		&+\displaystyle\int_{0}^{\theta}\|f(\nu,0,\int_{0}^{\nu}g(\nu,\tau,0)d\tau)\|+\displaystyle\frac{M_{B}}{\beta}\Big(\mathcal{L}_{4}+\mathcal{L}_{5}\|z\|_{C}\Big)a\Bigg)\\
 		&\leq M\overline{Y}+\displaystyle\frac{Ma^{q}}{\Gamma(q)}\Big(C_{1}k+C_{2}C_{3}k+a\|m\|_{L^{\infty}(J,\mathbb{R}^{+})}+\displaystyle\frac{M_{B}}{\beta}\Big(\mathcal{L}_{4}+\mathcal{L}_{5}k\Big)\Big)\\
 		&\leq \mathcal{L}+\displaystyle\frac{Ma^{q}}{\Gamma(q)}\Big(C_{1}+C_{2}C_{3}+\frac{M_{B}}{\beta}\mathcal{L}_{5}\Big)k
 	\end{array}
 \end{equation*}
	with $\mathcal{L}=M\overline{Y}+\displaystyle\frac{Ma^{q}}{\Gamma(q)}\Big(a\|m\|_{L^{\infty}(J,\mathbb{R}^{+})}+\frac{\mathcal{L}_{4}M_{B}}{\beta}\Big)$.
	Hence, for sufficiently large $k>0$, we obtain
	$\|\mathcal{G}_{1}z+\mathcal{G}_{2}w\|\leq k$, then  $\mathcal{G}_{1}z+\mathcal{G}_{2}w\in B_{k}$ for any $z,w\in B_{k}$.\\
	\textbf{Step II: } $(\mathcal{G}_{1})$ is a contraction on $B_{k}$. For any $\theta\in J$ and $z,w\in B_{k}$, we have
	\begin{equation*}
		\|(\mathcal{G}_{1}z)(\theta)-(\mathcal{G}_{1}w)(\theta)\|\leq \|C_{q}(\theta)\|\|\phi(z)-\phi(w)\|+\|S_{q}(\theta)\|\|\psi(z)-\psi(w)\|.
	\end{equation*}
	As a consequence of $(\mathcal{H}_{4})$, we obtain
	$
	\|(\mathcal{G}_{1}z)(\theta)-(\mathcal{G}_{1}w)(\theta)\|\leq Md_{1}\|z-w\|+Md_{2}\|z-w\|\leq M(d_{1}+d_{2})\|z-w\|
	$.
	Since $M(d_{1}+d_{2})<1$, we deduce that $\mathcal{G}_{1}$ is a contraction.\\
\textbf{Step III: } $\mathcal{G}_{2}$ is a completely continuous operator.
	First, we will show the continuity of the operator $\mathcal{G}_{2}$ on $B_{k}$. Let $(z_{n})$ be a sequence in $B_{k}$ such that $z_{n}\to z$ in $B_{k}$.  Then by $(\mathcal{H}_{2})$, $(\mathcal{H}_{3})$ and the fact that $z_{n}\to z$ we have
	\begin{equation*}
		f(\nu,z_{n},\int_{0}^{\nu}g(\nu,\tau,z_{n})d\tau)\to f(\nu,z,\int_{0}^{\nu}g(\nu,\tau,z)d\tau),\quad  \text{as}\,\,n\to+\infty.
	\end{equation*}
This combined with the dominated convergence theorem, we have
\begin{equation*}
	\begin{array}{lll}
		 \|(\mathcal{G}_{2}z_{n})(\theta)-(\mathcal{G}_{2}z)(\theta)\|&\leq\displaystyle\frac{Ma^{q-1}}{\Gamma(q)}\displaystyle\int_{0}^{\theta}\|f(\nu,z_{n},\int_{0}^{\nu}g(\nu,\tau,z_{n})d\tau)- f(\nu,z,\int_{0}^{\nu}g(\nu,\tau,z)d\tau)\|d\nu\\
		&+\displaystyle\frac{M_{B}}{\beta}\bigg(\displaystyle\frac{M_{B}M}{\Gamma(q)}a^{q-1}\bigg)^{2}\Bigg[T\|C_{q}(\theta)\|\|\phi(z_{n}-\phi(z))\|+T\|S_{q}(\theta)\|\|\psi(z_{n})-\psi(z)\|\\
		&+\displaystyle\frac{Ma^{q}}{\Gamma(q)}\displaystyle\int_{0}^{\theta}\|f(\sigma,z_{n},\int_{0}^{\sigma}g(\sigma,\tau,z_{n})d\tau)- f(\sigma,z,\int_{0}^{\sigma}g(\sigma,\tau,x)d\tau)\|d\sigma\Bigg]\\
		&\leq \displaystyle\frac{Ma^{q-1}}{\Gamma(q)}\displaystyle\int_{0}^{\theta}\|f(\nu,z_{n},\int_{0}^{\nu}g(\nu,\tau,z_{n})d\tau)- f(\nu,z,\int_{0}^{\nu}g(\nu,\tau,z)d\tau)\|d\nu\\
		&+\displaystyle\frac{M_{B}}{\beta}\bigg(\displaystyle\frac{M_{B}M}{\Gamma(q)}a^{q-1}\bigg)^{2}\Bigg[MT(d_{1}+d_{2})\|z_{n}-z\|
		+\displaystyle\frac{Ma^{q}}{\Gamma(q)}\displaystyle\int_{0}^{\theta}\|f(\sigma,z_{n},\int_{0}^{\sigma}g(\sigma,\tau,\\
		&\displaystyle z_{n})d\tau)- f(\sigma,z,\int_{0}^{\sigma}g(\sigma,\tau,z)d\tau)\|d\sigma\Bigg].
	\end{array}
\end{equation*}
 So, $\|(\mathcal{G}_{2}z_{n})-(\mathcal{G}_{2}z)\|\to 0$ as $n\to+\infty$, which implies the continuity of $\mathcal{G}_{2}$ in $B_{k}$.\\
	Now, we will prove  the compactness of the operator $\mathcal{G}_{2}$. To prove this, we first show that $\{(\mathcal{G}_{2}z)(\theta),~ z\in B_{k}\}$ is relatively compact in $Z$. Subsequently, we prove that $\{(\mathcal{G}_{2}z)(\theta),~ z\in B_{k}\}$ is uniformly bounded. For any $\theta\in [0, a]$ and $z\in B_{k}$, we have from lemma~\ref{lemmaM}, and lemma~\ref{lemmaB}
	\begin{equation*}
		\begin{array}{lll}
			 \|(\mathcal{G}_{2}z)(\theta)\|&\leq \displaystyle\frac{Ma^{q-1}}{\Gamma(q)}\Bigg(\displaystyle\int_{0}^{\theta}\|f(\nu,z(\nu),\int_{0}^{\nu}g(\nu,\tau,z(\tau))d\tau)-f(\nu,0,\int_{0}^{\nu}g(\nu,\tau,0)d\tau)\|d\nu\\
			&+\displaystyle\int_{0}^{\theta}\|f(\nu,0,\int_{0}^{\nu}g(\nu,\tau,0)d\tau)\|+\displaystyle\displaystyle\int_{0}^{\theta}\|Bu_{\beta}(\nu,z)\|d\nu\Bigg)\\	
			&\leq \displaystyle\frac{Ma^{q}}{\Gamma(q)}\Bigg(C_{1}k+C_{2}C_{3}k+T\|m\|_{L^{\infty}(J,\mathbb{R}^{+})}+\displaystyle\frac{M_{B}}{\beta}\Big(\mathcal{L}_{4}+\mathcal{L}_{5}k\Big)\Bigg)\\
			&=\displaystyle\frac{Ma^{q}}{\Gamma(q)}\Big(C_{1}+C_{2}C_{3}\frac{\mathcal{L}_{5}}{\beta}\Big)k+\displaystyle\frac{Ma^{q}}{\Gamma(q)}\Big(a\|m\|_{L^{\infty}(J,\mathbb{R}^{+})}+\displaystyle\frac{M_{B}}{\beta}\mathcal{L}_{4}\Big)
		\end{array}
	\end{equation*}
 Therefore, we have $\displaystyle\sup_{z\in B_{k}}\|(\mathcal{G}_{2}z)(\theta)\|<\infty$,
 which means that $\{\mathcal{G}_{2}z,~~ z\in B_{k}\}$ is uniformly bounded. Finally, we will show that $\{\mathcal{G}_{2}z,~~ z\in B_{k}\}$ is an equicontinuous family of functions on $[0, a]$. For any fixed $z\in B_{k}$ and $0\leq \theta_{1}<\theta_{2}\leq a$, we get
 \begin{equation*}
 	\begin{array}{llll}
 		\|(\mathcal{G}_{2}z)(\theta_{2})-(\mathcal{G}_{2}z)(\theta_{1})\|&\leq \Bigg\|\displaystyle\int_{0}^{\theta_{2}}P_{q}(\theta_{2}-\nu)f(\nu,z(\nu),\displaystyle\int_{0}^{\nu}g(\nu,\tau,z(\tau))d\tau)d\nu-\displaystyle\int_{0}^{\theta_{1}}P_{q}(\theta_{1}-\nu)f(\nu,z(\nu),\\
 		&\displaystyle\int_{0}^{\nu}g(\nu,\tau,z(\tau))d\tau)d\nu+\displaystyle\int_{0}^{\theta_{2}}P_{q}(\theta_{2}-\nu)Bu_{\beta}(\nu,z)d\nu-\displaystyle\int_{0}^{\theta_{1}}P_{q}(\theta_{1}-\nu)Bu_{\beta}(\nu,z)d\nu\Bigg\|\\
 		&\leq \displaystyle\int_{0}^{\theta_{1}}\|P_{q}(\theta_{2}-\nu)-P_{q}(\theta_{1}-\nu)\|\Big[\|f(\nu,z(\nu),\int_{0}^{\nu}g(\nu,\tau,z(\tau))d\tau)-f(\nu,0,\int_{0}^{\nu}g(\nu,\tau,0)d\tau)\|d\nu\\
 		&+\displaystyle\|f(\nu,0,\int_{0}^{\nu}g(\nu,\tau,0)d\tau)\|+\|Bu_{\beta}(\nu,z)\|\Big]d\nu\\
 		&+\displaystyle\int_{\theta_{1}}^{\theta_{2}}\|P_{q}(\theta_{2}-\nu)\|\Big[\|f(\nu,z(\nu),\displaystyle\int_{0}^{\nu}g(\nu,\tau,z(\tau))d\tau)\|+\|Bu_{\beta}(\nu,z)\|\Big]d\nu
 	\end{array}
 \end{equation*}
  By lemma~\ref{lemmaM},  inequality~\eqref{estim control}, $(\mathcal{H}_{2})-(ii)$ and $(\mathcal{H}_{3})-(ii)$ we get
 	\begin{equation*}
 		\begin{array}{llll}
 			 |(\mathcal{G}_{2}z)(\theta_{2})-(\mathcal{G}_{2}z)(\theta_{1})\|&\leq \displaystyle\int_{0}^{\theta_{1}}\|P_{q}(\theta_{2}-\nu)-P_{q}(\theta_{1}-\nu)\|\Big[C_{1}k+C_{2}C_{3}k+\|m\|_{L^{\infty}(J,\mathbb{R}^{+})}+\displaystyle\frac{M_{B}}{\beta}\Big(\mathcal{L}_{4}+\mathcal{L}_{5}\|z\|_{C}\Big)\Big]d\nu\\
 			&+\displaystyle\frac{Ma^{q-1}}{\Gamma(q)}\displaystyle\int_{\theta_{1}}^{\theta_{2}}\Big[C_{1}k+C_{2}C_{3}k+\|m\|_{L^{\infty}(J,\mathbb{R}^{+})}+\displaystyle\frac{M_{B}}{\beta}\Big(\mathcal{L}_{4}+\mathcal{L}_{5}\|z\|_{C}\Big)\Big]d\nu\\
 			&\leq \displaystyle\int_{0}^{\theta_{1}}\|P_{q}(\theta_{2}-\nu)-P_{q}(\theta_{1}-\nu)\|\Big[C_{1}k+C_{2}C_{3}k+\|m\|_{L^{\infty}(J,\mathbb{R}^{+})}+\displaystyle\frac{M_{B}}{\beta}\Big(\mathcal{L}_{4}+\mathcal{L}_{5}k\Big)\Big]d\nu\\
 			&+\frac{Ma^{q-1}}{\Gamma(q)}\Big[C_{1}k+C_{2}C_{3}k+\|m\|_{L^{\infty}(J,\mathbb{R}^{+})}+\displaystyle\frac{M_{B}}{\beta}\Big(\mathcal{L}_{4}+\mathcal{L}_{5}k\Big)\Big]|\theta_{2}-\theta_{1}|.
 		\end{array}
 	\end{equation*}
Since $P_{q}(\theta)$ is a compact operator for $\theta>0$, $P_{q}(\theta)$ is uniformly continuous. Then $\|(\mathcal{G}_{2}z)(\theta_{2})-(\mathcal{G}_{2}z)\|$ tends to zero independently of $z\in B_{k}$. This implies that the family of functions $\{\mathcal{G}_{2}z,~z\in B_{k}\}$ is equicontinuous. Then, by Arzela-Ascoli theorem, $\mathcal{G}_{2}$ is compact. Thus, by Theorem~\ref{Kras}, we conclude that there exists a fixed point $z(\cdot)$ for $\mathcal{G}=\mathcal{G}_{1}+\mathcal{G}_{2}$ on $B_{k}$. Therefore,  the control system \eqref{system1} has at least one mild solution on $[0, a]$. This completes the proof.
  \end{proof}  
 In the next section, we discuss the approximate controllability of the control system \eqref{system1} which is the main topic of this work.
  
 \subsection{Approximate controllability}\label{subsec34}

In this section, we present our main result on approximate controllability of the control system \eqref{system1}. For this purpose, we add the following hypothesis:  
 \begin{itemize}
     \item[$(\mathcal{H}_{5})$] The linear control system \eqref{linearsys} is approximately controllable.
     \item[$(\mathcal{H}_{6})$] The function $f$ is uniformly bounded on her domain.
 \end{itemize}
 
\begin{remark}\label{rem}
     In view of \cite{N.I} $(\mathcal{H}_{5})$ is equivalent to $\beta\mathbf{R}(\beta,\mathbf{K}_{0}^{a})=\beta(\beta I+\mathbf{K}_{0}^{a})^{-1}\to 0\quad as\quad \beta\to 0^{+}$ in the strong operator topology.
 \end{remark}
Now we are in a position to prove the approximate controllability of system \eqref{system1}. That is

\begin{theorem}\label{Thm appro}
	Assume that the hypothesis $(\mathcal{H}_{1})-(\mathcal{H}_{6})$ are satisfied and $C_{q}(\theta)$, $S_{q}(\theta)$ are compact. Then the system \eqref{system1} is approximately controllable on $[0, a]$.
\end{theorem}

\begin{proof}
	Let $z_{\beta}(\cdot)=z(\cdot, z_{0},z_{1},u_{\beta})$ be a fixed point of $\mathcal{G}_{1}+\mathcal{G}_{2}$ in $B_{k}$. By Theorem~\ref{Thm mild}, $z_{\beta}(\cdot)$ is a mild solution of \eqref{system1} associated to the control 
	\begin{equation*}
		\begin{array}{ll}
				u_{\beta}(\theta,z)&=B^{\star}P_{q}^{\star}(a-\theta)(\beta I+\mathbf{K}_{0}^{a})^{-1}\bigg(z_{d}-C_{q}(a)(z_{0}-\phi(z))-S_{q}(a)(z_{1}-\psi(z))\bigg) \\&-B^{\star}P_{q}^{\star}(a-\theta)\displaystyle\int_{0}^{a}(\beta I+\mathbf{K}_{0}^{a})^{-1}P_{q}(a-\nu)f(\nu,z(\nu),\int_{0}^{\nu}g(\nu,\tau,z(\tau))d\tau)d\nu. 
		\end{array}
	\end{equation*}	
Then, when $\theta=a$,
	\begin{equation}\label{eqxd}
		\begin{array}{llll}
			z_{\beta}(a)&=C_{q}(a)[z_{0}-\phi(z_{\beta})]+S_{q}(a)[z_{1}-\psi(z_{\beta})]+\displaystyle\int_{0}^{a}P_{q}(a-\nu)f(\nu,z_{\beta}\nu),\displaystyle\int_{0}^{\nu}g(\nu,\tau,z_{\beta}(\tau))d\tau)d\nu\\
			&+\displaystyle\int_{0}^{a}P_{q}(a-\nu)\Big[BB^{\star}P_{q}^{\star}(\theta-\nu)R(\beta,\mathbf{K}_{0}^{a})\Big(z_{d}-C_{q}(a)(z_{0}-\phi(z_{\beta}))-S_{q}(a)(z_{1}-\psi(z_{\beta}))\\
			&-\displaystyle\int_{0}^{a}P_{q}(a-\sigma)f(\sigma,z_{\beta}(\sigma),\int_{0}^{\sigma}g(\sigma,\tau,z_{\beta}(\tau))d\tau)d\sigma\Big)\Big]d\nu\\
			&=C_{q}(a)[z_{0}-\phi(z_{\beta})]+S_{q}(a)[z_{1}-\psi(z_{\beta})]+\displaystyle\int_{0}^{a}P_{q}(a-\nu)f(\nu,z(\nu),\displaystyle\int_{0}^{\nu}g(\nu,\tau,z_{\beta}(\tau))d\tau)d\nu\\
			&+(-\beta I+\beta I+\mathbf{K}_{0}^{a})R(\beta,\mathbf{K}_{0}^{a})\Big(z_{d}-C_{q}(a)(z_{0}-\phi(z_{\beta}))-S_{q}(a)(z_{1}-\psi(z_{\beta}))-\displaystyle\int_{0}^{a}P_{q}(a-\sigma)f(\sigma,z_{\beta}(\sigma),\\
			&\displaystyle\int_{0}^{\sigma}g(\sigma,\tau,z_{\beta}(\tau))d\tau)d\sigma\\
			&=z_{d}-\beta R(\beta,\mathbf{K}_{0}^{a})\Big(z_{d}-C_{q}(a)(z_{0}-\phi(z_{\beta}))-S_{q}(a)(z_{1}-\psi(z_{\beta}))-\displaystyle\int_{0}^{a}P_{q}(a-\sigma)f(\sigma,z_{\beta}(\sigma),\int_{0}^{\sigma}\\
			&g(\sigma,\tau,z_{\beta}(\tau))d\tau)d\sigma\Big)\\
			&=z_{d}-\beta R(\beta,\mathbf{K}_{0}^{a})p(z_{\beta}),
		\end{array}
	\end{equation}
where $p(z_{\beta})= z_{d}-C_{q}(a)(z_{0}-\phi(z_{\beta}))-S_{q}(a)(z_{1}-\psi(z_{\beta}))\displaystyle\int_{0}^{a}P_{q}(a-\sigma)f(\sigma,z_{\beta}(\sigma),\int_{0}^{\sigma}
		g(\sigma,\tau,z_{\beta}(\tau))d\tau)d\sigma$.\\
By the assumption $(\mathcal{H}_{6})$, we get that the sequence $\{f(\nu,z_{\beta}(\nu),\displaystyle\int_{0}^{\nu}g(\nu,\tau,z_{\beta}(\tau))d\tau)\}$ is uniformly bounded on $J$, consequently, there is a sequence still denoted by $\{f(\nu,z_{\beta}(\nu),\displaystyle\int_{0}^{\nu}g(\nu,\tau,z_{\beta}(\tau))d\tau)\}$, that converges weakly to $f^{\prime}$ in $L^{2}(J\times Z\times Z,Z)$. Denote $\omega=z_{d}-C_{q}(a)(z_{0}-\phi(z_{\beta}))-S_{q}(a)(z_{1}-\psi(z_{\beta}))\displaystyle\int_{0}^{a}P_{q}(a-\sigma)f^{\prime}d\sigma$. Then, we have
	\begin{equation}\label{estimation}
	\begin{array}{lll}
	\|p(z_{\beta})-\omega\|&=
	\Big\|\displaystyle\int_{0}^{\theta}P_{q}(a-\sigma)\Big[f(\sigma,z_{\beta}(\sigma),\int_{0}^{\sigma}g(\sigma,\tau,z_{\beta}(\tau))d\tau)-f^{\prime}\Big]d\sigma\Big\|\\
 &\leq \displaystyle\sup_{\theta\in [0, a]}\Big\|\displaystyle\int_{0}^{\theta}P_{q}(a-\sigma)\Big[f(\sigma,z_{\beta}(\sigma),\int_{0}^{\sigma}g(\sigma,\tau,z_{\beta}(\tau))d\tau)-f^{\prime}\Big]d\sigma\Big\|
	\end{array}
	\end{equation}
 Now, the compactness of the operator $P_{q}(\theta)$, implies that 
 $\|p(z_{\beta})-\omega\|\to 0$ as $\beta\to 0$  for all $\theta\in [0, a]$.
 Moreover, from \eqref{eqxd}, we obtain
 \begin{equation*}
 	\|z_{\beta}(a)-z_{d}\|\leq \|\beta R(\beta,\mathbf{K}_{0}^{a})(\omega)\|+\|\beta R(\beta,\mathbf{K}_{0}^{a})\|\|(p(z_{\beta})-\omega)\|
 	\leq \|\beta R(\beta,\mathbf{K}_{0}^{a})(\omega)\|+\|p(z_{\beta})-\omega)\|.
 \end{equation*}
By the assumption $(\mathcal{H}_{5})$, the approximate controllability of \eqref{linearsys} is equivalent to the convergence of the operator $\beta R(\beta,\mathbf{K}_{0}^{a})$ to zero as $\beta\to 0^{+}$ in the strong topology. Then, it follows by the previous remark in subsection~\ref{subsec34}	and the estimation \eqref{estimation} that $\|z_{\beta}(a)-z_{d}\|\to 0$ as  $\beta\to 0^{+}$ which implies the approximate controllability of the control system \eqref{system1}.
\end{proof}
\section*{ Example}\label{sec4}
Let $Z=L^{2}([0,\pi])$. Consider the following fractional control system
\begin{equation}\label{Example1}
	\begin{cases}
		^{c}D^{q}\vartheta(\theta,y) =\displaystyle\frac{\partial^{2}}{\partial y^{2}}\vartheta(\theta,y)+\displaystyle\frac{e^{-\theta}|\vartheta(\theta,y)|}{(3+e^{a})(1+|\vartheta(\theta,y)|)}+\displaystyle \int_{0}^{\theta}\displaystyle\frac{e^{\nu}}{\sqrt{2}+|\vartheta(\nu,y)|}d\nu +Bu(\theta,y), & \theta \in J, \\
		\displaystyle   \vartheta(\theta,0)=\vartheta(\theta,\pi)=0,\quad\quad \vartheta^{\prime}(\theta,0)=\vartheta^{\prime}(\theta,\pi)=0,\\
		\vartheta(\theta,0)+\sum_{i=1}^{i=p}a_{i
		}\vartheta(\theta_{i},y)=\vartheta_{0}(y)\quad\quad \vartheta^{\prime}(\theta,0)+\sum_{i=1}^{i=p}\gamma_{i}\vartheta(\theta_{i},y)=\vartheta_{1}(y),
	\end{cases}
\end{equation}
Where $^{c}D^{q}$ is the Caputo fractional derivative of order $q\in (1, 2)$, $\theta\in J=[0,1]$, $a_{i}, \gamma_{i}\in \mathbb{R}$, $p\in \mathbb{N}$. We assume that there exists $L_{1},~L_{2}> 0$ such that $\displaystyle\sum_{i=1}^{p}|a_{i}|<L_{1}$ and $\displaystyle\sum_{i=1}^{p}|\gamma_{i}|<L_{2}$. The operator $\mathcal{A}:D(\mathcal{A})\subset Z\to Z$ is defined by $
\mathcal{A}(\vartheta)=-\displaystyle\frac{\partial^{2}(\vartheta)}{\partial y^{2}}$, with the domain $ D(\mathcal{A})=\{\vartheta(\cdot)\in Z:\vartheta,\vartheta^{\prime}\,\text{are absolutely continuous}, \vartheta^{\prime\prime}\in Z, \vartheta(0)=\vartheta(\pi)=0\}$.
It is well known that the operator $-\mathcal{A}$ is densely defined in $Z$ and $-\mathcal{A}$ is the infinitesimal generator of a uniformly bounded strongly continuous cosine family $\{C(\theta),\,\theta\geq 0\}$, see \cite{Zhou, Arendt}.\\ Furthermore, $-\mathcal{A}$ has a discrete spectrum, the eigenvalues are $\lambda_{n}=-n^{2}\pi^{2},\,n\in \mathbb{N}$, with corresponding normalized eigenvectors $e_{n}(y)=\sqrt{\displaystyle\frac{2}{\pi}}sin(n\pi y),\,y\in [0,\pi],n\in \mathbb{N}$. It is clear that $\{e_{n}\}_{n\in \mathbb{N}}$ form an orthonormal basis of $X$. Then  
$
\mathcal{A}\vartheta=-\displaystyle\sum_{n=1}^{\infty}\lambda_{n}\langle \vartheta,e_{n}\rangle e_{n},\, \vartheta\in D(\mathcal{A}),
$
and the cosine function is given by
$
C(\theta)\vartheta=\displaystyle\sum_{n=1}^{\infty}cos(\sqrt{\lambda_{n}}\theta)\langle \vartheta,e_{n}\rangle e_{n},\, \vartheta\in Z,
$
and the associated sine function is given by
$
S(\theta)\vartheta=\displaystyle\sum_{n=1}^{\infty}\displaystyle\frac{1}{\sqrt{\lambda_{n}}}sin(\sqrt{\lambda_{n}}{\theta)}\langle \vartheta,e_{n}\rangle e_{n},\, \vartheta\in Z.
$
It is easy to check that for all $\theta\geq 0$, $S(\theta)$ is compact, $C(\theta)$ is strongly continuous and $\|C(\theta)\|\leq 1$.\\
The system \eqref{Example1} can be formulated as the following nonlocal problem in $Z$.
\begin{equation}\label{sys}
	\begin{cases}
		^{c}D^{q}z(\theta) =^{}-\mathcal{A} z(\theta)+ f(\theta,z(\theta),\displaystyle\int_{0}^{\theta}g(\theta,\nu,z(\nu))d\nu)+Bu(\theta), & \theta \in J=[0, 1], \\
		\displaystyle   z(0)+\phi(z)=z_{0},\\
		z^{\prime}(0)+\psi(z)=z_{1},
	\end{cases}
\end{equation}
where $[z(\theta)](y)=\vartheta(\theta,y)$, the function $g:J\times J\times Z\to Z$ is defined by
\begin{equation*}
	g(\theta,\nu,\varphi)=\displaystyle\frac{e^{\nu}}{\sqrt{2}+|\varphi(\nu,y)|},\quad y\in [0,\pi],
\end{equation*}
 the function $f:J\times X\times Z\to Z$ is given by
 \begin{equation*}
 	\begin{array}{ll}
 		f(\theta,\varphi,\displaystyle\int_{0}^{\theta}g(\theta,\nu,\varphi)d\nu)(y)=\displaystyle\frac{e^{-\theta}|\varphi(\theta,y)|}{(3+e^{\theta})(1+|\varphi(\theta,y)|)}+\displaystyle \int_{0}^{\theta}\displaystyle\frac{e^{\theta}}{\sqrt{2}+|\varphi(\nu,y)|}d\nu,\quad y\in [0,\pi],\\\\
 		\phi(z)(y)=\displaystyle\sum_{i=1}^{i=p}a_{i
 		}z(\theta)(y),\quad \quad\qquad \quad\qquad  \psi(z)(y)=\sum_{i=1}^{i=p}\gamma_{i
 		}z(\theta)(y)
 	\end{array}
 \end{equation*}
 We have demonstrated that the assumptions $(\mathcal{H}_{2})$-$(\mathcal{H}_{4})$ are satisfied.
Now, define an infinite dimensional space 
$$\mathcal{U}=\{u=\sum_{i=2}^{\infty}\langle u_{i},e_{i}\rangle e_{i}(y)\Big|\sum_{i=2}^{\infty}\langle u_{i},e_{i}\rangle^{2}<\infty\},$$ 
with a norm defined by $\|u\|=\left(\displaystyle\sum_{i=2}^{\infty}\langle u_{i},e_{i}\rangle^{2}\right)^{\frac{1}{2}}$ and a linear continuous mapping $B: \mathcal{U}\to Z$ given by
$$Bu=2u_{2}e_{1}(y)+\sum_{i=2}^{\infty}u_{i}e_{i}(y).$$ 
It is obvious that $u(\theta,y)=\displaystyle\sum_{i=2}^{\infty}\langle u_{i},e_{i}\rangle e_{i}(y)\in L^{2}(J,\mathcal{U})$, and  $Bu(\theta)=2u_{1}(\theta)e_{1}(y)+\displaystyle\sum_{i=2}^{\infty}\langle u_{i}(\theta),e_{i}\rangle e_{i}(y)\in L^{2}(J,Z)$.\\
Obviously, $\|B\|\leq \sqrt{5}$. 
the assumption $(\mathcal{H}_{5})$ is satisfied.
Thus the corresponding linear system to \eqref{Example1} is approximately controllable on $[0, 1]$. Then all the assumptions are satisfied. So by Theorem~\ref{Thm appro}, the system \eqref{Example1} is approximately controllable on $[0, 1]$.   

\section*{Conclusion}
In this work, fractional nonlinear evolution equations have been considered in Banach spaces. We obtain the approximate controllability for the system \eqref{system1} under different conditions. More precisely, conditions are formulated and proved under which approximate controllability of the nonlinear control system is implied by the approximate controllability of its corresponding linear part. In particular, the controllability problem is transformed into a fixed point problem for an appropraite nonlinear operator in a function space. Krasnoselskii's fixed point theorem and fractional calculations are used to demonstrate the existence of a fixed point of this operator and established the approximate controllability of the proposed system. This work opens to other questions, this is the case of the approximate controllability for impulsive neutral semilinear and nonlinear evolution equations with nonlocal conditions and with delay. This will be the purpose of future research. 


\end{document}